\documentclass{article}
\usepackage{amssymb,latexsym,times,amsmath,amsthm, graphicx, epsfig}
\usepackage[enableskew]{youngtab}
\usepackage{young}
\usepackage{tikz,pgf}
\usepackage{semtrans}

\newtheorem{theorem}{Theorem}
\newtheorem{corollary}{Corollary}
\newtheorem{lemma}{Lemma}

\newtheorem{definition}{Definition}

\begin{document}

\title{Schmidt-type theorems for partitions with uncounted parts}
\author{George E. Andrews, Pennsylvania State University \\ William J. Keith, Michigan Technological University}

\maketitle

\abstract{Schmidt's theorem is significantly generalized, to partitions in which periodic but otherwise arbitrary subsets of parts are counted or uncounted.  The identification of such sets of partitions with colored partitions satisfying certain specifications may be a generally useful tool for establishing sum-product $q$-series identities, examples of which are given.}

\section{Introduction}

In \cite{AndrewsPaule}, the first author and Peter Paule reproved and extended the following theorem of Schmidt \cite{Schmidt}:

\begin{theorem}\label{schmidt} Let $p(n)$ be the number of partitions $\lambda_1 + \lambda_2 + \dots$ of the integer $n$, and let $f(n)$ denote the number of partitions $\pi_1 + \pi_2 + \dots$ into distinct parts $\pi_i > \pi_{i+1}$ such that $n = \pi_1 + \pi_3 + \pi_5 + \dots$.  Then $f(n) = p(n)$ for all $n \geq 1$.
\end{theorem}

Andrews and Paule's reproof employed MacMahon's partition analysis, noting that this is a very natural approach to such theorems.  They then extended this theorem to partition diamonds with Schmidt type conditions on the places summed, proving generating functions and congruences for such partitions.  Their surfacing of Schmidt's idea has led to a spate of  investigations (\cite{BerkovichUncu}, \cite{Ji}, \cite{LiYee}) considering this concept in the context of modern partition theory.

Our goal in this paper is to extend and refine Schmidt's theorem in a different direction.  Our main theorem, proved bijectively using a colored version of the map of Stockhofe, is a great generalization of the partitions considered and the places summed.

\begin{theorem}\label{maintheorem} Fix $m > 2$.  Let $S = \{ s_1, s_2, \dots, s_i \} \subseteq \{1, 2, \dots, m-1 \}$ with $1 \in S$, and $\vec{\rho} = (\rho_1, \dots, \rho_{m-1})$. Denote by $P_{m,S}(n;\vec{\rho})$ the number of partitions $\lambda = (\lambda_1, \dots, )$ into parts repeating less than $m$ times in which \begin{align*} n &= \sum_{c \equiv s_j \pmod{m}} \lambda_c \\ \rho_k &= \sum_{c \equiv k \pmod{m}} \lambda_k - \lambda_{k+1}.\end{align*}  Then $P_{m,S}(n;\vec{\rho})$ is also equal to the number of partitions of $n$ where parts $k \pmod{i}$ appear in $s_{k+1} - s_k$ colors, letting $s_{i+1} = m$, and, labeling colors of parts $k \pmod{i}$ by $s_k$ through $s_{k+1}-1$, parts of color $j$ appear $\rho_j$ times.
\end{theorem}

The sheer number of parameters in this theorem perhaps obfuscates the relationship to Schmidt, so we offer the following corollaries where the generalization may be clearer:

\begin{corollary}\label{firsti} Fix $m > 2$ and $1 \leq i < m$.  Let $R_{m,i}(n)$ be the set of partitions $\lambda_1 + \lambda_2 + \dots $, $\lambda_1 \geq \lambda_2 \geq \dots > 0$,  in which parts can only appear fewer than $m$ times, and in which $n = \sum_{s=0}^\infty \sum_{j=1}^i \lambda_{sm+i}$.  Then $$\sum_{n=0}^\infty R_{m,i}(n)q^n = \frac{1}{(q;q)_\infty {(q^i;q^i)_\infty}^{m-i-1}}.$$
\end{corollary}

Schmidt's original theorem is the $(m,i) = (2,1)$ case of this corollary.  Its refinement by numbers of parts is:

\begin{corollary}\label{f21} Let $f_{2,1}(n,m)$ denote the number of partitions in $R_{2,1}$, i.e. partitions $\lambda_1 + \lambda_2 + \dots$, $\lambda_1 > \lambda_2 > \lambda_3 > \dots > 0$, in which $$n = \lambda_1 + \lambda_3 + \lambda_5 + \dots,$$ with the additional condition that $$m = \lambda_1 - \lambda_2 + \lambda_3 - \lambda_4 + \lambda_5 - \lambda_6 + \dots .$$ Then $$f_{2,1}(n,m) = p(n,m).$$
\end{corollary}

\noindent \textbf{Remark:} In fact, Corollary \ref{f21} can be obtained as a consequence of Mork's original proof of Schmidt's Theorem.

The first nontrivial case not covered by the previous corollaries is the following.

\begin{corollary} The number of partitions in which parts repeat less than 4 times and $$n = \lambda_1 + \lambda _3 + \lambda_5 + \dots$$ is equal to the number of partitions in which even parts appear in one color and odd parts appear in two colors.
\end{corollary}

One observes that the latter set is well-known to have the same generating function as the much-studied \emph{overpartitions}.  Hence, this theorem gives another combinatorial set equivalent to overpartitions, furnished with its own statistics.

Cases of the $q$-series identities implied by this theorem can be stated as sum-product identities, the validity of which immediately follows once the generating functions for the respective sets are established.  Tools for establishing these identities, such as the method of weighted words (\cite{AlladiGordon}, \cite{Dousse}), are of considerable utility in the theory of $q$-series.  We give several examples in this paper.

The following two theorems concern the modulus 3.

\begin{theorem}\label{f32} Denote by $p(n,m)$ the number of ordinary partitions of $n$ into exactly $m$ parts.  Let $f_{3,2}(n,m)$ denote the number of partitions in $R_{3,2}$, i.e. partitions $\lambda_1 + \lambda_2 + \dots$, $\lambda_1 \geq \lambda_2 \geq \lambda_3 \geq \dots > 0$, into positive parts repeating not more than twice, in which $$n = \lambda_1 + \lambda_2 + \lambda_4 + \lambda_5 + \lambda_7 + \lambda_ 8 + \dots,$$ with the additional condition that $$m = \lambda_1 - \lambda_3 + \lambda_4 - \lambda_6 + \lambda_7 - \lambda_9 + \dots .$$ Then $$f_{3,2}(n,m) = p(n,m).$$ This equality is equivalent to the truth of the identity $$\sum_{n \geq j \geq 0} \frac{(-1)^j q^{n^2-n+\binom{j+n+1}{2}}(-q;q)_{n-j} \left[ {n \atop j} \right]_{q^2} t^{j+n}}{(tq;q)_{2n} (q^2;q^2)_n} = \frac{1}{(tq;q)_\infty}.$$
\end{theorem}

\noindent \textbf{Example:} The partitions counted by $p(7,3) = 4$ are $5+1+1$, $4+2+1$, $3+3+1$, and $3+2+2$.  The partitions counted by $f_{3,2}(7,3)$ are $5+2+2$, $4+3+1$, $4+2+2+1$, and $3+3+1+1$.

\begin{theorem}\label{f31} Let $p(n;m_1,m_2)$ be the number of 2-colored partitions in which $m_1$ parts are of the first color and $m_2$ parts are of the second color.  Let $f_{3,1}(n;m_1,m_2)$ denote the number of partitions in $R_{3,1}$, i.e. $\lambda_1 + \lambda_2 + \dots$, $\lambda_1 \geq \lambda_2 \geq \lambda_3 \geq \dots > 0$, into positive parts repeating not more than twice, in which $$n = \lambda_1 + \lambda_4 + \lambda_7 + \dots$$ and $$m_1 = \lambda_1 - \lambda_2 + \lambda_4 - \lambda_5 + \lambda_7 - \lambda_8 + \dots ,$$ $$m_2 = \lambda_2 - \lambda_3 + \lambda_5 - \lambda_6 + \lambda_8 - \lambda_9 + \dots .$$ Then $$f_{3,1}(n;m_1,m_2) = p(n;m_1,m_2).$$  This equality is equivalent to the truth of the identity \begin{multline*}\sum_{N \geq 0} \sum_{{j+k \geq N} \atop {j, k \leq N}} \frac{t_1^j t_2^k (-1)^{j+k-N} q^{\binom{N}{2} + \binom{j+1}{2} + \binom{k+1}{2}} \left[ {N \atop {N-j, N-k, j+k-N}} \right]_q}{(t_1q;q)_N (t_2q;q)_N (q;q)_N} \\ = \frac{1}{(t_1q;q)_\infty(t_2q;q)_\infty}.\end{multline*}
\end{theorem}

We further generalize Theorem \ref{f32} to any modulus and give the associated $q$-series identity.

\begin{theorem}\label{genf} Fix $k>1$.  Denote by $p(n,m)$ the number of ordinary partitions of $n$ into exactly $m$ parts.  Let $f_{k,k-1}(n,m)$ denote the number of partitions in $R_{k,k-1}$, i.e. partitions $\lambda_1 + \lambda_2 + \dots$, $\lambda_1 \geq \lambda_2 \geq \lambda_3 \geq \dots > 0$, into positive parts repeating less than $k$ times, in which $$n = \sum_{i=0}^\infty \lambda_i - \sum_{i=0}^\infty \lambda_{ik},$$ with the additional condition that $$m = \lambda_1 - \lambda_k + \lambda_{k+1} - \lambda_{2k} + \lambda_{2k+1} - \lambda_{3k} + \dots .$$ Then $$f_{k,k-1}(n,m) = p(n,m).$$ This equality is equivalent to the truth of the identity \begin{multline*}\sum_{n \geq 0} \frac{ \sum_{i=n}^{(k-1)n} t^i \sum_{j=0}^n (-1)^{n+i+j} q^{(k-1) (\binom{n-j}{2} + i j ) + \binom{i+1}{2}} \left[ {{(k-1)(n-j)} \atop i} \right]_q \left[ {n \atop j} \right]_{q^{k-1}} }{(tq;q)_{{(k-1)}n} (q^{k-1};q^{k-1})_n} \\ = \frac{1}{(tq;q)_\infty}.
\end{multline*}
\end{theorem}

In Section \ref{backnote} we give the notation we shall use in this paper.  In Section \ref{proofs} we prove the main theorem, which proves the partition statements of Theorems \ref{f32}, \ref{f31}, and \ref{genf}.  In Section \ref{refinements} we prove the $q$-series claims of Theorems \ref{f32}, \ref{f31}, and \ref{genf}.  Finally, in Section \ref{further} we discuss potential extensions and the desirability of alternative proofs.

\phantom{.}

\noindent \textbf{Acknowledgement:} During the preprint phase of drafting this paper, we were contacted by Isaac Konan, who has produced a broad generalization of Stockhofe's map along the lines of Bressoud's energy concept or the method of weighted words, capable of handling large sets of colors with a variety of conditions: see Theorem 2.8 in \cite{Konan}.   In this paper we require only Stockhofe's original map, slightly generalized to permit addition of a particular color in each case; certainly Konan's theorem, with sufficient specializations of colored parts, would imply our main bijective result Thereom \ref{maintheorem} as well.  It is possible that Konan's algorithms applied to a broader class of Schmidt-type conditions produce an even more general theorem and associated $q$-series identities; what these may be, we leave to the reader interested in the attempt.

\section{Background and notation}\label{backnote}

The main bijection we use is a colored generalization of a map arising in the Ph.D. thesis of Dieter Stockhofe \cite{Stockhofe}, which is in German, with an English translation to be found as an appendix to the thesis of the second author, \cite{KeithThesis}.  

Throughout, we treat partitions as weakly decreasing sequences of nonnegative integers $\lambda = (\lambda_1, \dots, \lambda_k)$, with finitely many positive part sizes $\lambda_i \in \mathbb{N}$, and with infinite tails of zeroes suppressed in the notation.

The Ferrers diagram of a partition $\lambda$ is a set of boxes in the fourth quadrant wherein the row with bottom right corners at $(-i,j)$ is populated with $\lambda_i$ boxes.  The \emph{conjugate} $\lambda^{\prime}$ of $\lambda$ is the partition which has Ferrers diagram that is the reflection of this diagram across the line $y=-x$.

\phantom{.}

\noindent \textbf{Example:} The Ferrers diagrams of the partition $(4,2,2,1)$ and its conjugate $(4,3,1,1)$ are illustrated below.

$$\young(\hfil\hfil\hfil\hfil,\hfil\hfil::,\hfil\hfil::,\hfil::) \quad \quad \quad \young(\hfil\hfil\hfil\hfil,\hfil\hfil\hfil:,\hfil:::,\hfil:::)
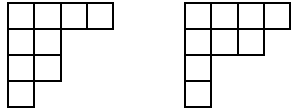
$$

\phantom{.}

\begin{definition} A partition $\lambda = (\lambda_1, \dots, \lambda_k)$ is $m$-\emph{regular} if no part size $\lambda_i$ is divisible by $m$, $m$-\emph{distinct} if no part size repeats $m$ or more times, and $m$-\emph{flat} if differences $\lambda_i - \lambda_{i+1} < m$ for all $i$ (we emphasize that this includes the final positive part).
\end{definition}

It is well known that these three classes are equinumerous by a variety of bijections.  Classically the map from $m$-flat to $m$-regular partitions was performed first by conjugating $\lambda$ and then employing one of many maps from $m$-distinct to $m$-regular partitions, such as Glaisher's.  Stockhofe produced a direct map.  In \cite{KeithXiong} Keith and Xiong employed the map to produce statements matching sets of partitions with desired alternating sum types and length types.

To define our map we need two vector operations on partitions: 

\begin{align*} n\lambda &= (n \lambda_1 , n \lambda_2 , \dots ) \\
\lambda + \mu &= (\lambda_1 + \mu_1, \lambda_2 + \mu_2, \dots ).
\end{align*}

We may now describe Stockhofe's bijection $\phi$ from $m$-flat to $m$-regular partitions.

The following fact is easy to prove:

\begin{lemma} Let $\vec{v} = (v_1, \dots, v_k)$ be a sequence of nonzero residues modulo $m$.  Then there is a unique partition $\lambda(\vec{v})$ which is $m$-regular, $m$-flat, and for which $\lambda({\vec{v}})_i \equiv v_i \pmod{m}$.
\end{lemma}

\phantom{.}

\noindent \textbf{Example:} If $m=5$ and $\vec{v} = (2,3,1,4,4,2)$, then $\lambda(\vec{v}) = (12,8,6,4,4,2)$.

\phantom{.}

It is now easy to see that any $m$-regular partition may be written uniquely as $\lambda(\vec{v}) + m \mu$ for $\vec{v}$ defined by its residue sequence mod $m$, and some $\mu$.  Each part of $\mu^{\prime}$ is no larger than the number of parts in $\lambda(\vec{v})$.  The main insight behind the construction of the map $\phi$ is showing that $\mu$ can be retrieved from an $m$-flat partition in a way that leaves $\lambda(\vec{v})$.

\phantom{.}

\noindent \textbf{Algorithm for } $\phi$: let $\lambda$ be an $m$-flat partition.  Initialize $\mu = ()$, the empty partition.

\phantom{.}

\noindent \textbf{Step 1.} Working from the smallest to the largest part, remove from $\lambda$ any parts divisible by $m$ for which, after removal, the partition is still $m$-flat.  Append these parts to $\mu$.  These will be parts such that

\begin{itemize}
\item $\lambda_i = k m = \lambda_{i-1}$, i.e. all but the first of a repeated part divisible by $m$; 
\item $\lambda_1 = k m$, i.e. the largest part is divisible by $m$; OR
\item parts $\lambda_i = k_i m$, $i>1$, such that $\lambda_{i-1} = k_i m+j_1$, $\lambda_{i+1} = (k_i-1)m + j_2$, with $0 < j_1 < j_2 < m$.
\end{itemize}

The latter, in other words, are multiples of $m$ between two nonmultiples of $m$ that differ by less than $m$.

Call the remaining partition $\lambda^-$.

The remaining parts divisible by $m$ in $\lambda^-$ are all distinct, not the largest (or smallest) parts, and any remaining part $\lambda_i = k_i m$ lies between $\lambda_{i-1} = k_i m + j_1$ and $\lambda_{i+1} = (k_i-1)m+j_2$ with $0 < j_2 \leq j_1 < m$.  Hence it is possible to leave an $m$-flat partition by removing $\lambda_i$ and also subtract $m$ from every larger part.

Working now from the largest to the smallest remaining parts in $\lambda^-$: 

\noindent \textbf{Step 2:} For each remaining part $\lambda_i = k_i m$, remove $\lambda_i$ and subtract $m$ from all parts $\lambda_j$, $j < i$.  Append $\lambda_i + (i-1)m$ to $\mu$ as a part.

After all of these parts are removed, the remaining partition is $m$-flat and $m$-regular, and hence is automatically $\lambda(\vec{v})$ for the vector of nonzero residues of the original partition.  Now construct $$\phi(\lambda) = \lambda(\vec{v}) + m \mu^{\prime}.$$

It is established in \cite{Stockhofe} that any $\mu$ with parts no larger than the number of parts nonzero modulo $m$ can arise from this removal process, and that reversing the process is unique.  That is, there is a unique way to insert a part size reversing Step 2; the reversal of Step 1 for smaller part sizes is immediate.  Hence this is a bijection between the sets of $m$-flat and $m$-regular partitions.

\phantom{.}

\noindent \textbf{Example:} Let $m=5$, $\lambda = (26, 25, 22, 19, 15, 13, 11, 7, 5, 5,3)$.  Set $\mu = ()$ to start.  In Step 1, we remove 5, 5 and 25.  We now have $$\mu = (25,5,5) \quad , \quad \lambda^- = (26, 22, 19, 15, 13, 11, 7, 3).$$

In Step 2 we remove 15, and three more copies of 5 from larger parts.  We append 30 to $\mu$.  We now have $$\mu = (30, 25, 5,5) = 5 (6, 5, 1) \quad , \quad \lambda(\vec{v}) = (21,17,14,13,11,7,3).$$

Finally, conjugate $\mu$ and add.  We obtain \begin{align*}\phi(\lambda) &=  \lambda(\vec{v}) + 5 (3,2,2,2,2,2,1) \\ &= (21,17,14,13,11,7,3) + (20, 10, 10, 10, 10, 5) \\ &= (41, 27, 24, 23, 21, 12, 3).\end{align*}

We will need to generalize this map to colored parts.  Colored integers are simply integers with subscripts: $$\mathbb{N}_j = \{ n_i \vert n \in \mathbb{N}, 1 \leq i \leq j \}.$$

The $n$ are the sizes of the colored integers, and the subscripts $i$ their color.  One may consider partitions into colored parts by summing the sizes of the colored parts, once an order has been assigned the subscripts.  We will order subscripts to weakly decrease among parts of equal size.

When dealing with colored parts, there is not a general convention for the partition vector operations.

For the purpose of this paper, we will take scalar multiplication to be defined for an uncolored scalar, and the resulting parts to have the colors of the corresponding part of $\lambda$, or a colored scalar and an uncolored $\lambda$, and the resulting parts to all have the color of the scalar.

For partition addition, we will only ever be adding summands $\mu$ having a single color which in fact we will wish to avoid; therefore we will assign the color of each part of the sum to be the color of the corresponding part of $\lambda$.

In Section \ref{refinements} we will employ $q$-series techniques to prove sum-product identities associated to two specific cases of the main theorem.  These theorems employ the following standard notation.

The $q$-Pochhammer symbols are \begin{align*}(a;q)_\infty &= \prod_{i=0}^\infty (1-aq^i) \\ (a;q)_n &= \frac{(aq^n;q)_\infty}{(a;q)_\infty}.\end{align*}  From these we may construct the $q$-multinomial coefficients $$\left[ {{B_1+B_2+\dots+B_k} \atop {\{B_1,B_2,\dots,B_k\}}} \right]_q = \frac{(q;q)_{B_1+B_2+\dots+B_k}}{(q;q)_{B_1}(q;q)_{B_2} \dots (q;q)_{B_k}}.$$  A more common notation when $k=2$ in this definition is the $q$-binomial coefficient $$\left[ {A \atop B} \right]_q = \frac{(q;q)_A}{(q;q)_{A-B}(q;q)_B}.$$

The $q$-hypergeometric series is defined by \begin{multline*}{}_r \phi_s \left[ \begin{matrix} a_1, & a_2, & \dots, & a_r \\ b_1, & \dots, & b_s & \end{matrix};q,z \right] \\ = \sum_{n=0}^\infty \frac{(a_1;q)_n \dots (a_r;q)_n}{(b_1;q)_n \dots (b_s;q)_n} \left( (-1)^n q^{\binom{n}{2}} \right)^{1+r-s} z^n.\end{multline*}

Two useful theorems we will need are the $q$-Chu-Vandermonde summation of the special case $${}_2 \phi_1 \left[ \begin{matrix} q^{-n}, & b \\ & c \end{matrix};q, cq^n/b \right] = \frac{(c/b;q)_n}{(c;q)_n}$$ \noindent or equivalently $${}_2\phi_1 \left( \begin{matrix} q^{-n}, & b \\ & c \end{matrix};q;q \right) = \frac{b^n (c/b;q)_n}{(c;q)_n}$$ \noindent and its special case the $q$-Vandermonde summation $$ \sum_j \left[ {m \atop {k-j}} \right]_q \left[ {n \atop j} \right]_q q^{j(m-k+j)} = \left[ {{m+n} \atop k} \right]_q,$$ and finally the $q$-binomial theorem, $$\sum_{n=0}^\infty z^n q^{\binom{n+1}{2}} \left[ {N \atop n} \right]_q = (-zq;q)_N.$$

\section{Proof of Theorem \ref{maintheorem}}\label{proofs}

Fix $m$ and the set of places $S$ mod $m$ to be counted.  In conjugate, the condition of $m$-distinctness becomes $m$-flatness, and now each part size is assigned a weight depending on the number of counted places it intersects.  This is best illustrated by example.

\phantom{.}

Let $m = 5$ and set $S = \{1, 2, 3\}$.  That is, parts must differ by less than 5, and we shall count only places 1, 2, or 3 mod 5.  Consider the partition $$\lambda = (11,11,11,10,10,8,8,7,7,7,7,6,6,5,5,4,4,4,4,3,3,3,2,2,2,1).$$

Conjugate this partition, keeping track of counted parts: the Ferrers diagram of $\lambda^\prime$, with counted parts being filled, is as follows.

$$\young(\blacksquare\blacksquare\blacksquare\hfil\hfil\blacksquare\blacksquare\blacksquare\hfil\hfil\blacksquare\blacksquare\blacksquare\hfil\hfil\blacksquare\blacksquare\blacksquare\hfil\hfil\blacksquare\blacksquare\blacksquare\hfil\hfil\blacksquare,\blacksquare\blacksquare\blacksquare\hfil\hfil\blacksquare\blacksquare\blacksquare\hfil\hfil\blacksquare\blacksquare\blacksquare\hfil\hfil\blacksquare\blacksquare\blacksquare\hfil\hfil\blacksquare\blacksquare\blacksquare\hfil\hfil:,\blacksquare\blacksquare\blacksquare\hfil\hfil\blacksquare\blacksquare\blacksquare\hfil\hfil\blacksquare\blacksquare\blacksquare\hfil\hfil\blacksquare\blacksquare\blacksquare\hfil\hfil\blacksquare\blacksquare::::,\blacksquare\blacksquare\blacksquare\hfil\hfil\blacksquare\blacksquare\blacksquare\hfil\hfil\blacksquare\blacksquare\blacksquare\hfil\hfil\blacksquare\blacksquare\blacksquare\hfil:::::::,\blacksquare\blacksquare\blacksquare\hfil\hfil\blacksquare\blacksquare\blacksquare\hfil\hfil\blacksquare\blacksquare\blacksquare\hfil\hfil:::::::::::,\blacksquare\blacksquare\blacksquare\hfil\hfil\blacksquare\blacksquare\blacksquare\hfil\hfil\blacksquare\blacksquare\blacksquare:::::::::::::,\blacksquare\blacksquare\blacksquare\hfil\hfil\blacksquare\blacksquare\blacksquare\hfil\hfil\blacksquare:::::::::::::::,\blacksquare\blacksquare\blacksquare\hfil\hfil\blacksquare\blacksquare:::::::::::::::::::,\blacksquare\blacksquare\blacksquare\hfil\hfil:::::::::::::::::::::,\blacksquare\blacksquare\blacksquare\hfil\hfil:::::::::::::::::::::,\blacksquare\blacksquare\blacksquare:::::::::::::::::::::::)$$

It will be observed that, other than the coloration, this is the partition of our example in Section \ref{backnote}, $$\lambda = (26, 25, 22, 19, 15, 13, 11, 7, 5, 5,3).$$

If we count the weights contributed by each part of $\lambda^\prime$, we have $$\vert \lambda^\prime \vert = 16 + 15 + 14 + 12 + 9 + 9 + 7 + 5 + 3 + 3 + 3 = 96.$$  It is clear that effective part sizes that are nonzero mod 3 can arise in exactly one way, whereas multiples of 3 can arise in up to 3 different ways, depending on whether they have zero, one, or two uncounted boxes at the end of the row.

Counting the reduced weight of the partition given by the specified parts and calculating the differences among part sizes in places mod 5, we find that $\lambda$ is counted in $$P_{5,\{1,2,3\}}(96,(2,2,2,1)).$$

Applying $\phi$, recall that we obtained $$\phi(\lambda) = (41, 27, 24, 23, 21, 12, 3).$$  Reduce the weight by counted parts: $3 \lfloor \lambda_i / 5 \rfloor + r$, where $r$ is the appropriate additional colored residue mod 3.  Give reduced parts 1 or 2 mod 3 the subscripts 1 or 2 (which were their original residues mod 5), and give reduced multiples of 3 the subscripts 3 or 4 according to whether the original part was 3 or 4 mod 5.  We obtain that the colored reduction of this partition is $$(25_1, 17_2, 15_4, 15_3, 13_1, 8_2, 3_3).$$  As desired, there are respectively 2, 2, 2, and 1 parts of colors one through four.

\phantom{.}

\begin{proof}[Proof of the theorem] Fix $m$ and $S$ as in the theorem.  When we conjugate an $m$-distinct partition and assign to each part the color of the residue of its natural size mod $m$, the vector counting the differences by place residue mod $m$ becomes exactly the count of parts of each color appearing.  The number of colors available for a part of size $j \pmod{i}$ is exactly 1 more than the number of uncounted boxes that may arise after $j \pmod{i}$ are filled, which is in turn $s_{j+1} - s_j$, with the exception of parts that are multiple of $m$, which will be removed by $\phi$.

If we apply $\phi$ with modulus $m$ to an $m$-flat partition, we observe that adding or subtracting multiples of $m$ to or from a part does not change the color of the part when colored according to this scheme.  This is the reason for our convention on addition of colored parts stated earlier.

The total reduced weight does not change as multiples of $m$ are removed from any one part and added to others in units of $m$ at a time.

Thus the properties of the map immediately imply the equality of partitions in the two sets.
\end{proof}

\phantom{.}

\noindent \textbf{Remark:} Indeed, the bijection preserves not only the total counts but the ordered residue-vector, which is even more precise than the theorem as stated; however, that would have been an even more elaborate statement.  The interested reader may certainly construct the implication from the proof above.

\phantom{.}

\noindent \textbf{Remark:} Andrews and Paule in \cite{AndrewsPaule} also proved that, if one considers the set of arbitrary partitions with no restrictions on parts and only adds parts in odd places, one obtains two-colored partitions.  This is even easier to prove from the colored-conjugate viewpoint: simply observe that in the conjugate, arbitrary partitions are possible and each part size can appear in two different colors.  The places of parts to be counted can be arbitrarily generalized and the resulting part sizes listed, additional colors in a part size occurring for each uncounted place.

Kathy K. Q. Ji in \cite{Ji} gave a different combinatorial proof of that theorem, which lends itself to additional statistical refinements.

\section{Proofs of $q$-series identities}\label{refinements}

\subsection{Proof of Theorem \ref{f32}}\label{f32proof}

We first deal with the $(m,i) = (3,2)$ case, stating that partitions into parts repeating no more than twice with parts counted only in places not divisible by 3, are in bijection with ordinary partitions into exactly $m$ parts when $$m = \lambda_1 - \lambda_3 + \lambda_4 - \lambda_6 + \lambda_7 - \lambda_9 + \dots .$$

Denote $$X_i = x_1 x_2 \dots x_i t^{\chi_3(i)},$$ where $\chi_3(i) = 0 $ if $3 \vert i$ and 1 otherwise.  Let $$P_N = P_N(x_1, x_2, x_3, \dots , x_N;t)$$ be the generating function for partitions with zeros allowed, of the form $$\lambda_1 + \lambda_2 + \dots + \lambda_N \quad \quad (\lambda_1 \geq \lambda_2 \geq \lambda_3 \geq \dots \geq \lambda_N \geq 0)$$ where no part (including 0) appears more than twice, the exponent of $x_i$ is $\lambda_i$, and the exponent of $t$ is $$\lambda_1 - \lambda_3 + \lambda_4 - \lambda_6 + \lambda_7 - \lambda_9 + \dots .$$

A little combinatorial thought gives that the $P_N$ satisfy the recurrence 

\begin{equation}\label{pnrecursion}
P_N = \frac{X_{N-2} P_{N-2}}{1-X_N} + \frac{X_{N-1} P_{N-1}}{1-X_N}, 
\end{equation}

\noindent with $P_0 = 1$, $P_1 = \frac{1}{1-X_1}$, and $P_2 = \frac{1}{(1-X_1)(1-X_2)}$.  This recursion and the initial conditions completely define $P_N$.

\phantom{.}

\noindent \textbf{Remark:} As an aside, while the recurrence is natural once seen, we suggest that the tools of Partition Analysis can be useful in establishing such statements.

\phantom{.}

From (\ref{pnrecursion}) we see that we may write $$P_N = \frac{\pi_N}{\prod_{i=1}^N (1-X_i)}$$ where $\pi_N$ is a polynomial in the $X_i$.  The $\pi_N$ are defined by the initial conditions and recursion $\pi_0 = \pi_1 = \pi_2 = 1$ and for $N>2$, 

\begin{equation}\label{pirecursion}
\pi_N = (1-X_{N-1})X_{N-2} \pi_{N-2} + X_{N-1} \pi_{N-1}.
\end{equation}

We now make the substitution $$x_i = \begin{cases} q &\text{ if } 3 \nmid i \\ 1 &\text{ if } 3 \vert i.\end{cases}$$

This makes $$X_i = q^{\lceil \frac{2i}{3} \rceil}t^{\chi_3(i)}.$$  In doing so, we see that the function $P_N(q,q,1,q,q,1,\dots;t)$ has the coefficient of $q^nt^m$ counting the number of partitions into exactly $N$ nonnegative parts, in which any part size including 0 must appear less than 3 times, with $n = \lambda_1 + \lambda_2 + \lambda_4 + \lambda_5 + \dots$, and in which $m = \lambda_1 - \lambda_3 + \lambda_4 - \lambda_7 + \dots$.

We now give the numerator $\pi_N$.

\begin{lemma}\label{pi3} For $n \geq 0$ and $r \in \{0,1,2\}$, with the above substitutions for the $X_i$, $$\pi_{3n+r}(q,q,1,\dots;t) = \sum_{j=0}^n (-1)^j q^{n^2-n+r(n+j)+\binom{j+n+1}{2}}(-q;q)_{n-j} \left[ {n \atop j}\right]_{q^2} t^{j+n}.$$
\end{lemma}

\begin{proof} We see by inspection that the formula claimed gives 1 for $\pi_0$, $\pi_1$, and $\pi_2$, as required.  It remains to show that the formulas satisfy the recursion (\ref{pirecursion}).

In the three cases of interest, after substitution recursion (\ref{pirecursion}) becomes

\begin{align*}
\pi_{3n+2}(q,q,1,\dots) &= (1-tq^{2n+1}) q^{2n} \pi_{3n} + tq^{2n+1} \pi_{3n+1} \\
\pi_{3n+1}(q,q,1,\dots) &= (1-q^{2n}) tq^{2n} \pi_{3n-1} + q^{2n} \pi_{3n} \\
\pi_{3n}(q,q,1,\dots) &= (1-tq^{2n}) tq^{2n-1} \pi_{3n-2} + tq^{2n} \pi_{3n-1}.
\end{align*}

Now for each case compare coefficients of $t^k$ on both sides.

\end{proof}

Now the denominator of $P_{3N}(q,q,1,\dots;t)$ is $$(tq;q)_{2N}(q^2;q^2)_N$$ and the numerator is $\pi_{3N}$.  

So far from previous work we have that $$\sum_{n,m \geq 0} f_{3,2}(n,m) q^n t^m = \sum_{n \geq 0} P_{3N} (q,q,1,q,q,1,\dots;t).$$  We note that we need only sum over indices $3N$ since $P_{3N}$ will count once any partition of the type counted by $f_{3,2}(n,m)$ into $3N$, $3N-1$, or $3N-2$ positive parts, and so all lengths of positive parts are covered.

The two-variable generating function for $p(n,m)$ is $$\sum_{n,m \geq 0} p(n,m) q^n t^m = \frac{1}{(tq;q)_\infty}.$$

Hence the equality of the Schmidt-type and the colored partitions stated by the main theorem yields the equivalence of the two expressions in sum:

$$\sum_{n \geq j \geq 0} \frac{(-1)^j q^{n^2-n+\binom{j+n+1}{2}}(-q;q)_{n-j} \left[ {n \atop j} \right]_{q^2} t^{j+n}}{(tq;q)_{2n} (q^2;q^2)_n} = \frac{1}{(tq;q)_\infty}.$$

\phantom{.}

\noindent \textbf{Remark:} At this point we observe that the $q$-series identity is completely proved.  The main theorem establishes the equality of the finest subdivisions appearing on either side, which when summed yield the two expressions.  One can proceed generally in this fashion if one has a set of colored partitions to deal with and wishes to prove an associated sum-product identity: attach the colored partitions to a Schmidt-type set, set up the associated recurrence, and hopefully solve.

\phantom{.}

The following work proves the identity by $q$-series analysis, independently of the bijection.  In this and the more complicated next case, the lemmas required and methods employed have some interest in their own right.

We further require the following result.

\begin{lemma}\label{2phi1} For integer $M \geq 0$, $$\sum_{j \geq 0} \frac{(q^{-M};q^2)_j(q^{-M+1};q^2)_jq^{2Mj-j^2}(-1)^j}{(q^2;q^2)_j} = q^{\binom{M}{2}}.$$
\end{lemma}

\begin{proof} The sum in question is $$\lim_{\tau \rightarrow 0} {}_2\phi_1 \left({{q^{-M},q^{-M+1};q^2,q^{2M}\tau^{-1}} \atop {q/\tau \phantom{ \quad ; q/\tau}}} \right) = q^{\binom{M}{2}},$$ by the $q$-Chu-Vandermonde summation (with cases split by whether $M$ is even or odd).
\end{proof}

Multiply both sides by $(tq;q)_\infty$, shift $n \rightarrow n+j$, and simplify to obtain that we may equivalently show $$\sum_{n,j \geq 0} \frac{(-1)^j q^{(n+j)^2-(n+j)+\binom{2j+n+1}{2}}t^{n+2j}(tq^{2n+2j+1};q)_\infty}{(q^2;q^2)_j(q;q)_n} = 1.$$

Expand the infinite product in the numerator with Euler's theorem: $$(tq^{2n+2j+1};q)_\infty = \sum_{m \geq 0} (-1)^m q^{\binom{m+1}{2}+m(2n+2j)}t^m.$$

The desired identity is thus equivalent to

$$\sum_{m,n,j \geq 0} \frac{(-1)^{j+m} q^{(n+j)^2-(n+j)+\binom{2j+n+1}{2}+\binom{m+1}{2}+m(2n+2j)}t^{m+n+2j}}{(q^2;q^2)_j(q;q)_n(q;q)_m} = 1.$$

Certainly the coefficient of $t^0$ on the left-hand side is 1, so it remains to show that the coefficient of $t^N$ is 0 for all $N >0$ to complete the proof.

Set $m = N-n-2j$; the coefficient of $t^N$ is then

\begin{align*}
& \phantom{=} (-1)^N q^{\binom{N+1}{2}} \sum_{n,j \geq 0} \frac{(-1)^{n+j}q^{Nn-n+j^2-j}}{(q^2;q^2)_j(q;q)_n(q;q)_{N-n-2j}} \\
&= (-1)^N q^{\binom{N+1}{2}} \sum_{n \geq 0}\frac{(-1)^{n} q^{Nn-n}}{(q;q)_n(q;q)_{N-n}} \lim_{\tau \rightarrow 0} {}_2\phi_1 \left( {{q^{-N+n},q^{-N+n+1};q^2,q^{2(N-n)}\tau^{-1}} \atop {q/\tau \phantom{ \quad ; q/\tau}}} \right) \\
&= (-1)^N q^{\binom{N+1}{2}} \sum_{n \geq 0} \frac{(-1)^n q^{Nn-n+\binom{N-n}{2}}}{(q;q)_n(q;q)_{N-n}} \quad \text{ (by Lemma \ref{2phi1}) } \\
&= \frac{(-1)^N q^{N^2}}{(q;q)_N} \sum_{n\geq 0} (-1)^n q^{\binom{n}{2}} \left[ {N \atop n} \right]_q \\
&= \frac{(-1)^N q^{N^2}}{(q;q)_N}(1;q)_N = 0,
\end{align*}

\noindent where the last line follows from the $q$-binomial theorem, as long as $N > 0$.

\hfill $\Box$

\subsection{Proof of Theorem \ref{f31}.} Recall that we are now showing the equivalence of $p(n;m_1,m_2)$, the number of partitions of $n$ into two colors of which $m_1$ parts are the first color and $m_2$ parts are the second color, and $f_{3,1}(n;m_1,m_2)$, counting the number of partitions $\lambda_1 + \lambda_2 + \dots$, $\lambda_1 \geq \lambda_2 \geq \lambda_3 \geq \dots $, into positive parts repeating not more than twice, in which $$n = \lambda_1 + \lambda_4 + \lambda_7 + \dots$$ and $$m_1 = \lambda_1 - \lambda_2 + \lambda_4 - \lambda_5 + \lambda_7 - \lambda_8 + \dots ,$$ $$m_2 = \lambda_2 - \lambda_3 + \lambda_5 - \lambda_6 + \lambda_8 - \lambda_9 + \dots .$$ 

Some of the machinery of the proof is analogous to the previous.  We here note the major differences.

Denote $Y_i = x_1 x_2 \dots x_i t(i)$, where $$t(i) = \begin{cases} t_1 &\text{ if } i \equiv 1 \pmod{3} \\ t_2 &\text{ if } i \equiv 2 \pmod{3} \\ 1 &\text{ if } i \equiv 0 \pmod{3}.\end{cases}$$

Let $Q_N(x_1,x_2,\dots ; t_1,t_2)$ generate such partitions with $N$ nonnegative parts.  Then $$Q_N = \frac{Y_{N-2} Q_{N-2}}{1-Y_N} + \frac{Y_{N-1} Q_{N-1}}{1-Y_N}$$ with $Q_0=1$, $Q_1 = \frac{1}{1-Y_1}$, $Q_2 = \frac{1}{(1-Y_1)(1-Y_2)}$.

We find that $Q_N(q,1,1,q,1,1,\dots ; t_1,t_2) = \frac{\kappa_N}{\prod_{i=1}^N (1 - q^{\lceil i/3 \rceil} t(i) )}$ and show the following lemma:

\begin{lemma} We have \begin{multline*}\kappa_{3N}(q,t_1,t_2) = \\ \sum_{{j+k \geq N} \atop {j, k \leq N}} t_1^j t_2^k (-1)^{j+k-N} q^{\binom{N}{2} + \binom{j+1}{2} + \binom{k+1}{2}} \left[ {N \atop {N-j, N-k, j+k-N}} \right]_q\end{multline*} and $$\kappa_{3N+1}(q,t_1,t_2) = \kappa_{3N}(q,t_1q,t_2) \quad, \quad \kappa_{3N+2} = \kappa_{3N}(q,t_1q, t_2q).$$
\end{lemma}

The proof is again by verifying the necessary recurrence.

This having been done, we find that we have (or, when proceeding by $q$-series, would wish to show) the identity following. 

\begin{multline*}\sum_{N \geq 0} \sum_{{j+k \geq N} \atop {j, k \leq N}} \frac{t_1^j t_2^k (-1)^{j+k-N} q^{\binom{N}{2} + \binom{j+1}{2} + \binom{k+1}{2}} \left[ {N \atop {N-j, N-k, j+k-N}} \right]_q}{(t_1q;q)_N (t_2q;q)_N (q;q)_N} \\ = \frac{1}{(t_1q;q)_\infty(t_2q;q)_\infty}.\end{multline*}

Again we prove the above identity independently; the work is rather more involved than the previous theorem, and the methodology may have some independent utility. 

After certain substitutions and simplifications, and noting that the coefficient of $t_1^0 t_2^0$ is 1 on both sides, we find that the above statement is equivalent to the claim that, if either $u_1 > 0$ or $u_2 > 0$ and the other is at least 0, then

$$\sum_{N \geq 0} \sum_{j=0}^{\text{min}(N,u_1)} \sum_{k = N-j}^{\text{min}(N,u_2)} \frac{(-1)^N q^{\binom{N}{2} + (N-j)(u_1-j) + (N-k)(u_2-k)}}{(q)_{N-j}(q)_{u_1-j}(q)_{j+k-N}(q)_{N-k}(q)_{u_2-k}} = 0.$$

An analytic note is in order.  Recall that we defined $$\frac{1}{(q;q)_n} = \frac{(q^{n+1};q)_\infty}{(q;q)_\infty}.$$ This means that $(q;q)_n^{-1}$ is defined for all integers $n$, and in particular has value 0 if $n<0$ (i.e., the 0 term is in the numerator).  Thus we may write the desired identity as 

$$\sum_{N=0}^{u_1+u_2} \sum_{j=0}^{u_1} \sum_{k = 0}^{u_2} \frac{(-1)^N q^{\binom{N}{2} + (N-j)(u_1-j) + (N-k)(u_2-k)}}{(q)_{N-j}(q)_{u_1-j}(q)_{j+k-N}(q)_{N-k}(q)_{u_2-k}} = 0.$$

This in turn is the $t = u_1 + u_2$ case of a more general identity, 

\begin{multline}\label{sumA}\sum_{N=0}^{t} \sum_{j=0}^{u_1} \sum_{k = 0}^{u_2} \frac{(-1)^N q^{\binom{N}{2} + (N-j)(u_1-j) + (N-k)(u_2-k)}}{(q)_{N-j}(q)_{u_1-j}(q)_{j+k-N}(q)_{N-k}(q)_{u_2-k}} \\ = (-1)^t q^{\binom{t+1}{2}} \left[ {{u_1 + u_2 -1} \atop t} \right]_q \frac{1}{(q)_{u_1}(q)_{u_2}}.\end{multline}

In order to establish identity (\ref{sumA}), we prove the following lemma.

\begin{lemma}\label{lemmab} If $u_1, u_2 \geq 0$ and at least one of $u_1$ and $u_2$ is positive, then 
\begin{multline*}\sum_{j=0}^{u_1} \sum_{k = 0}^{u_2} \frac{q^{(t-j)(u_1-j) + (t-k)(u_2-k)}}{(q)_{t-j}(q)_{u_1-j}(q)_{j+k-t}(q)_{t-k}(q)_{u_2-k}} \\ = \left[ {{u_1 + u_2} \atop t} \right]_q \frac{1}{(q)_{u_1} (q)_{u_2}} . \end{multline*}
\end{lemma}

\noindent \emph{Proof of identity (\ref{sumA}) from the lemma:} Observe that identity (\ref{sumA}) holds in the $t=0$ case since both sides are $1/((q)_{u_1} (q)_{u_2})$.  If one now takes the diffference of identity (\ref{sumA}) evaluated at $t$, minus (\ref{sumA}) evaluated at $t-1$, one obtains Lemma \ref{lemmab} multiplied through by $(-1)^t q^{\binom{t}{2}}$.  Hence summation of Lemma \ref{lemmab} times $(-1)^t q^{\binom{t}{2}}$ evaluated from 0 to any desired $t$ gives us the identity.

\begin{proof}[Proof of the Lemma]
Since $t$ is a constant, we may take $k$ as the outer index of summation and, for each $k$, reindex $j \rightarrow j+t-k$.  We obtain

\begin{multline*}
\sum_{j=0}^{u_1} \sum_{k = 0}^{u_2} \frac{q^{(t-j)(u_1-j) + (t-k)(u_2-k)}}{(q)_{t-j}(q)_{u_1-j}(q)_{j+k-t}(q)_{t-k}(q)_{u_2-k}} \\ = \sum_{k=0}^{u_2} \frac{q^{(t-k)(u_2-k)}}{(q)_{t-k}(q)_{u_2-k}} \times \sum_{j \geq 0} \frac{q^{(k-j)(u_1-j-t+k)}}{(q)_{k-j}(q)_{u_1-t+k-j}(q)_j}.
\end{multline*}

Now consider the $j$-indexed sum, holding $k$, $u_1$, and $t$ as constants.  We find the following.

\begin{multline*}
\sum_{j \geq 0} \frac{q^{(k-j)(u_1-j-t+k)}}{(q)_{k-j}(q)_{u_1-t+k-j}(q)_j} = \frac{1}{(q)_k(q)_{u_1-t+k}} \\ \times \sum_{j\geq 0} \frac{(q^{-k})_jq^{kj-\binom{j}{2}}(q^{-(u_1-t+k)})_jq^{(u_1-t+k)j-\binom{j}{2}}}{(q)_j} \times q^{(k-j)(u_1-j-t+k)} \\
= \frac{q^{ku_1-kt+k^2}}{(q)_k(q)_{u_1-t+k}} \sum_{j \geq 0} \frac{(q^{-k})_j(q^{-(u_1-t+k)})_jq^j}{(q)_j} \\
= \frac{q^{ku_1-kt+k^2+k(-u_1+t-k)}}{(q)_k(q)_{u_1-t+k}} = \frac{1}{(q)_k(q)_{u_1-t+k}}.
\end{multline*}

The summation at the beginning of the last line follows from the $q$-Chu-Vandermonde summation $${}_2\phi_1 \left({{q^{-N},b;q,q} \atop {c \phantom{ \quad ; q/\tau}}} \right) = \frac{b^N (c/b)_N}{(c)_N}.$$

With this summed, our original expression becomes

\begin{multline*}
\sum_{k=0}^{u_2} \frac{q^{(t-k)(u_2-k)}}{(q)_{t-k}(q)_{u_2-k}(q)_k(q)_{u_1-t+k}} \\
= \frac{q^{tu_2}}{(q)_t(q)_{u_2}(q)_{u_1-t}} \sum_{k=0}^{u_2} \frac{q^{-ku_2-kt+k^2}(q^{-t})_k(q^{-u_2})_k}{(q)_k(q^{u_1-t+1})_k} \times q^{tk+u_2k-2\binom{k}{2}} \\
= \frac{q^{tu_2}}{(q)_t(q)_{u_2}(q)_{u_1-t}} {}_2\phi_1 \left({{q^{-t},q^{-u_2};q,q} \atop {q^{u_1-t+1} \phantom{ \quad ; q/\tau}}} \right) \\
= \frac{q^{tu_2}}{(q)_t(q)_{u_2}(q)_{u_1-t}} \frac{q^{-tu_2}(q^{u_1+1})_{u_2}}{(q^{u_1-t+1})_{u_2}} \\
= \frac{(q^{u_1+1})_{u_2}}{(q)_t(q)_{u_2}(q)_{u_1+u_2-t}} \\
= \frac{(q)_{u_1+u_2}}{(q)_t(q)_{u_2}(q)_{u_1+u_2-t}(q)_{u_1}} = \left[ {{u_1+u_2} \atop t} \right]_q \frac{1}{(q)_{u_1}(q)_{u_2}}.
\end{multline*}

With all the desired implications in hand, the original claim follows, and the theorem is proved.
\end{proof}


\subsection{Proof of Theorem \ref{genf}}

In this section we simply establish the $q$-series describing both sets equated in the theorem, and thereby immediately obtain the identity claimed.

We employ the notation of subsection \ref{f32proof}, and the same logical machinery, now expanded to accommodate the additional parameter $k$.  We write $$\sum_{n,m \geq 0} f_{k,k-1}(n,m) q^n t^m = \sum_{N \geq 0} P_{kN} (q,\dots,1,q,\dots,1,\dots;t).$$  We write $$P_n = \frac{\pi_n}{\prod_{i=1}^n (1-X_i)}$$ and have that the $\pi_n$ satisfy the recurrence \begin{multline*}\pi_n = X_{n-1} \pi_{n-1} + X_{n-2} (1-X_{n-1}) \pi_{n-2} + \dots \\ + X_{n-k+1} (1-X_{n-1})\dots (1-X_{n-k+2}) \pi_{n-k+1}.\end{multline*}

In order to reduce this recurrence somewhat, observe that all terms after the first also appear in the recurrence expanding $\pi_{n-1}$, with an additional factor of $(1-X_{n-1})$, as does one extra term.  Thus we may write \begin{align*} \pi_n &= X_{n-1} \pi_{n-1} + (1-X_{n-1}) \pi_{n-1} \\ & \quad \quad + X_{n-k} (1-X_{n-1})\dots (1-X_{n-k+1}) X_{n-k} \\
&= \pi_{n-1}+ X_{n-k} (1-X_{n-1})\dots (1-X_{n-k+1}) X_{n-k} . \end{align*}

Now for $n \geq 0$ and $0 \leq r < k$, write $$\pi_{kN+r} = \sum_{i=N}^{(k-1)N} t^i c(kN+r, i).$$  The truth of the bounds on $i$ is immediate for $N=0$ and will follow from analysis of the recurrence.  Although we again only require $\pi_{kN}$ to write $f_{k,k-1}$, we require the intermediate values $\pi_{kN+r}$ in order to solve the recurrence.

To prove the identity in the theorem, we will show the following lemma.

\begin{lemma}\label{genlemma} For $N \leq i \leq (k-1)N$, we have $$c(kN+r,i) = \sum_{j=0}^N (-1)^{N+i+j} q^{(k-1) ( \binom{N-j}{2}+ i j ) + \binom{i+1}{2} + i r} \left[{{(k-1)(N-j)} \atop i} \right]_q \left[ {N \atop j} \right]_{q^{k-1}}.$$
\end{lemma}

\begin{proof} If $N=0$ the formula gives the correct value 1, and so we assume inductively that the lemma holds for all smaller values of $kN+r$.  The summation behaves slightly differently depending on whether $r=0$ or not, so we consider the two cases separately.

Let $r>0$.  Then, specializing the $X_i$, the claim becomes \begin{multline*} \sum_{i=N}^{(k-1)N} c(kN+r,i) t^i = \sum_{i=N}^{(k-1)N} c(kN,i) q^{i(r-1)} t^i \\ - t q^{(k-1)(N-1)+r} (1-t q^{(k-1)N+r-1})\dots (1-t q^{(k-1)N+1})(1-q^{(k-1)N}) \\ \times (1-t q^{(k-1)(N-1)+k-1}) \dots (1-t q^{(k-1) (N-1)+r+1}) \\ \times \sum_{i=N-1}^{(k-1)(N-1)} c(k(N-1),i) q^{i r} t^i. 
\end{multline*}

We observe that the bounds on the powers of terms $t^i$ match, as claimed, and now wish to verify that the coefficient of any $t^i$ sums on the right hand side to the desired expression on the left.  

Observe the factor \begin{multline*} \sum_b t^b h_b(q) :=  (1-t q^{(k-1)N+r-1})\dots (1-t q^{(k-1)N+1}) \\ \times (1-t q^{(k-1)(N-1)+k-1}) \dots (1-t q^{(k-1) (N-1)+r+1}) . \end{multline*}  We have that $h_b(q)$ counts (with a weight of $(-1)^b$) partitions into exactly $b$ distinct parts chosen from among $(k-1)(N-1) + r + 1$ to $(k-1)(N-1) + r + (k-2),$ i.e. $b$ copies of $(k-1)(N-1) + r$ plus partitions into $b$ distinct parts ranging from 1 to $k-2$.

Extracting the coefficients of each $t^i$ and invoking the induction hypothesis, we find that for given $k$, $N$, $r$, and $i$, we wish to verify the equality

\begin{multline*} c(kN+r,i) = c(kN,i) q^{i(r-1)} - (1-q^{(k-1)N}) q^{(k-1)(N-1)+r} \\ \times \sum_{b=0}^{k-2} (-1)^b q^{b((k-1)(N-1)+r)+\binom{b+1}{2}} \left[ {{k-2} \atop b} \right]_q c(k(N-1),i-1-b) q^{r(i-1-b)} \\
=  c(kN,i) q^{i(r-1)} - (1-q^{(k-1)N}) q^{(k-1)(N-1)+ri} \\ \times \sum_{b=0}^{k-2} (-1)^b q^{b(k-1)(N-1)+\binom{b+1}{2}} \left[ {{k-2} \atop b} \right]_q c(k(N-1),i-1-b).
\end{multline*}

Inserting the claimed formulas, we thus find that we wish to verify the following equality of sums in $q$:

\begin{multline*} 
\sum_{j=0}^N (-1)^{N+i+j} q^{(k-1)( \binom{N-j}{2} + i j ) + \binom{i+1}{2} + i r} \left[ {{(k-1)(N-j)} \atop i} \right]_q \left[ {N \atop j} \right]_{q^{k-1}} \\
= \sum_{j=0}^N (-1)^{N+i+j} q^{(k-1)( \binom{N-j}{2} + i j ) + \binom{i+1}{2} + i (r-1)} \left[ {{(k-1)(N-j)} \atop i} \right]_q \left[ {N \atop j} \right]_{q^{k-1}} \\
- (1-q^{(k-1)N}) q^{(k-1)(N-1)+ir} \sum_{b=0}^{k-2} (-1)^b q^{b(k-1)(N-1)+\binom{b+1}{2}} \left[ {{k-2} \atop b} \right]_q \\ \times \sum_{\ell=0}^{N-1} (-1)^{N-1+i-1-b+\ell} q^{(k-1)( \binom{N-1-\ell}{2}+(i-1-b)\ell)+\binom{i-b}{2}} \left[ {{(k-1)(N-1-\ell)} \atop {i-1-b}} \right]_q \left[ {{N-1} \atop \ell} \right]_{q^{k-1}}.
\end{multline*}

Clearing a few common factors and gathering like terms, we reduce the desired identity to

\begin{multline*} 
(1-q^i) \sum_{j=0}^N (-1)^{j} q^{(k-1)( \binom{N-j}{2} + i j ) + \binom{i+1}{2}} \left[ {{(k-1)(N-j)} \atop i} \right]_q \left[ {N \atop j} \right]_{q^{k-1}} \\
= (1-q^{(k-1)N}) q^{(k-1)(N-1)+i} \\ \times \sum_{\ell=0}^{N-1} \sum_{b=0}^{k-2} (-1)^\ell q^{(k-1)(b(N-1)+ \ell(i-1-b) + \binom{N-1-\ell}{2}) + \binom{b+1}{2} + \binom{i-b}{2}} \\ \times \left[ {{k-2} \atop b} \right]_q  \left[ {{(k-1)(N-1-\ell)} \atop {i-1-b}} \right]_q \left[ {{N-1} \atop \ell} \right]_{q^{k-1}}.
\end{multline*}

We now employ the $q$-Vandermonde summation on the innermost sum in the index $b$, after pulling forward of the sum all factors not involving $b$.  Note that within this sum, $k$, $i$, $N$, and $\ell$ are fixed.  We have \begin{multline*} \sum_{b=0}^{k-2} q^{b (k-1)(N-1-\ell) +\binom{b+1}{2}+\binom{i-b}{2}} \left[ {{k-2} \atop b} \right]_q \left[ {{(k-1)(N-1-\ell)} \atop {i-1-b}} \right]_q \\ = q^{\binom{i}{2}} \sum_{b=0}^{k-2} q^{b((k-1)(N-1-\ell)-i+b)} \left[ {{k-2} \atop b} \right]_q \left[ {{(k-1)(N-1-\ell)} \atop {i-1-b}} \right]_q \\ = q^{\binom{i}{2}} \left[ {{k-2+(k-1)(N-1-\ell)} \atop {i-1}} \right]_q = q^{\binom{i}{2}} \left[ {{(k-1)(N-\ell)-1} \atop {i-1}} \right]_q. \end{multline*}

We have reduced the desired identity to 

\begin{multline*}
(1-q^i) \sum_{j=0}^N (-1)^j q^{(k-1)(\binom{N-j}{2}+ij)+\binom{i+1}{2}} \left[ {{(k-1)(N-j)} \atop i} \right]_q \left[ {N \atop j} \right]_{q^{k-1}} \\ = (1-q^{(k-1)N}) q^{(k-1)(N-1)+\binom{i+1}{2}} \sum_{\ell=0}^{N-1} (-1)^\ell q^{(k-1)(\binom{N-1-\ell}{2} + (i-1) \ell)} \\ \times \left[ {{N-1} \atop \ell} \right]_{q^{k-1}} \left[ {{(k-1)(N-\ell)-1} \atop {i-1}} \right]_q.
\end{multline*}

Now verification is just a little algebra.  Divide out the factor of $(1-q^{(k-1)N})$ from the right hand side and out of $\left[ {N \atop j} \right]_{q^{k-1}}$ on the left; cancel an additional factor $1-q^{(k-1)(N-j)}$ from the denominator of that symbol and the numerator of $\left[ {{(k-1)(N-j)} \atop i} \right]_q$ to leave $\left[ {{N-1} \atop j} \right]_{q^{k-1}}$; pull in the factor of $1-q^i$ to make the other $q$-binomial $\left[ {{(k-1)(N-\ell)-1} \atop {i-1}} \right]_q$.  Remove $(N-1-j)$ from the binomial under $(k-1)$ in exponent of the power of $q$ in front of the terms on the left hand side to make $\binom{N-1-j}{2}$, and use it to reduce the other summand to $(i-1)j$.

Observe that the $q$-binomial $\left[ {A \atop {A+1}} \right]_q$ = 0, so that the term $j=N$ contributes 0 and can be removed, and we have verified equality.  The claim of the theorem holds for the case $\pi_{kN+r}$, $0 < r < k$.

The logic for $\pi_{kN}$ is similar; we note the significant differences in the verification below.

The required identity is 

\begin{multline*} 
\sum_{j=0}^N (-1)^{N+i+j} q^{(k-1)( \binom{N-j}{2} + i j ) + \binom{i+1}{2}} \left[ {{(k-1)(N-j)} \atop i} \right]_q \left[ {N \atop j} \right]_{q^{k-1}} \\
= \sum_{j=0}^{N-1} (-1)^{N-1+i+j} q^{(k-1)( \binom{N-1-j}{2} + i j ) + \binom{i+1}{2} + i (k-1)} \left[ {{(k-1)(N-1-j)} \atop i} \right]_q \left[ {{N-1} \atop j} \right]_{q^{k-1}} \\
- q^{(k-1)(N-1)} \sum_{b=0}^{k-1} (-1)^b q^{b(k-1)(N-1)+\binom{b+1}{2}} \left[ {{k-1} \atop b} \right]_q \\ \times \sum_{\ell=0}^{N-1} (-1)^{N-1+i-b+\ell} q^{(k-1)( \binom{N-1-\ell}{2}+(i-b)\ell)+\binom{i-b+1}{2}} \left[ {{(k-1)(N-1-\ell)} \atop {i-b}} \right]_q \left[ {{N-1} \atop \ell} \right]_{q^{k-1}}.
\end{multline*}

After cancelling several common terms and applying the $q$-Vandermonde identity to the sum in $b$ again, we make the notationally convenient substitution $j=J-1$ to obtain that we wish to verify the following identity:

\begin{multline*} 
\sum_{j=0}^{N} (-1)^{j} q^{(k-1)( \binom{N-j}{2} + i j )} \left[ {{(k-1)(N-j)} \atop i} \right]_q \left[ {N \atop j} \right]_{q^{k-1}} \\
= \sum_{J=1}^{N} (-1)^{J} q^{(k-1)( \binom{N-J}{2} + i J )} \left[ {{(k-1)(N-J)} \atop i} \right]_q \left[ {{N-1} \atop {J-1}} \right]_{q^{k-1}} \\
+ q^{(k-1)(N-1)} \sum_{\ell=0}^{N-1} (-1)^{\ell} q^{(k-1)( \binom{N-1-\ell}{2}+i\ell} \left[ {{(k-1)(N-\ell)} \atop {i}} \right]_q \left[ {{N-1} \atop \ell} \right]_{q^{k-1}}.
\end{multline*}

Now, combine indices $\ell$ and $J$ as a single index $j$ and note that the $j=0$ and $j=N$ terms on the right match the corresponding terms on the left.  For $1 \leq j \leq N-1$, employ the $q$-Pascal identity $$\left[ {{N-1} \atop {j-1}} \right]_{q^{k-1}} + q^{(k-1) j} \left[ {{N-1} \atop j} \right]_{q^{k-1}} = \left[ {N \atop j} \right]_{q^{k-1}}.$$

Both sides now match and the theorem is proved.

\end{proof}

\section{Further work}\label{further}

One immediate potential route of further investigation would be generalization.  Theorem \ref{genf} is concerned with the sets $R_{n,n-1}$.  A natural question arising is the general $q$-series identity associated to all sets $R_{n,i}$ for all valid $i$.

On the combinatorial side, this technique can be developed much further.  For instance, there is no need to restrict ourselves to not counting parts.  Let $$\vec{w} = (w_1, w_2, w_3, \dots )$$ be a sequence of integer weights, and consider partitions in a set ${\mathcal{D}}$ in which the size statistic associated to $\lambda = (\lambda_1, \dots , \lambda_k)$ is $$\vert \lambda \vert = \vec{w} \cdot \vec{\lambda} = w_1 \lambda_1 + w_2 \lambda_2 + w_3 \lambda_3 + \dots .$$

In this paper we have been dealing only with $\vec{w}$ for which the entries are 0 and 1, and are periodic mod $m$.  But as long as $\vec{w}$ has the combinatorial property that any initial segment satisfies $$w_1 + w_2 + \dots + w_j > 0,$$ then the resulting distribution of sizes has finite coefficients and is thus combinatorially meaningful.  What theorems might result from examining the ensemble of such linear functions on partition sets?

\phantom{.}

\noindent \textbf{Example:} Let $\vec{w} = (1,2,3,2,1,1,2,3,2,1,\dots)$, repeating with period 5.  Then in conjugate one sees that parts arising have values 1, 3, 6, 8, and 9 mod 9.  If we restrict the weighted partitions to have parts repeating less than 5 times and employ the bijection $\phi$, we obtain the following theorem:

\begin{theorem} Partitions of $n$ into parts that are 1, 3, 6, or 8 mod 9 are equinumerous with partitions of $n$ into parts repeating less than 5 times in which parts are counted with weight $\vec{w} = (1,2,3,2,1,\dots)$.
\end{theorem}

Readers familiar with the Kanade-Russell conjectures \cite{KR} will recognize the product side of one of the remaining open conjectures, widely considered rather difficult at the present time.  Ka\u{g}an Kur\c{s}ug\"{o}z \cite{Kursungoz} has produced positive generating functions for the difference sides; if analysis of the theorem above yields a manipulable generating function which can be shown equal to his function in the relevant instance, the original conjecture would be proved.

\phantom{.} 

Mork's original combinatorial proof of Schmidt's theorem involved a placement of hooks which can also be shown to give the refinement by parts proven here. In fact it is not too difficult to see the following theorem:

\begin{theorem} Among partitions into distinct parts, the subpartition $(\lambda_1, \lambda_3, \lambda_5,\dots)$ which sums to $n$ arises exactly as many times as this same sequence is the list of hooklengths on the diagonal of partitions of $n$: there are $\lambda_{2j-1} - \lambda_{2j+1}-1$ possible positionings of a hook of size $\lambda_{2j-1}$ given any valid positioning of a hook of size $\lambda_{2j+1}$, and also this many possible values of $\lambda_{2j}$, plus 1 for the smallest hook.
\end{theorem}

Theorems concerning the enumeration of hooks are a subject of current interest in combinatorics.  Ji's map \cite{Ji} is one such, for Andrews and Paule's theorem on two-colored partitions.  The $i = m-1$ cases of Corollary \ref{firsti} all state the equivalence of ordinary partitions, and partitions into parts repeating $m$ times in which places 0 mod $m$ are not counted.  Is there a hook-like object whose placements are counted by the potential entries, which specifies a partition?  What other properties would these statistics have?

\end{document}